\documentclass[english,3p]{elsarticle}
\usepackage[T1]{fontenc}
\usepackage[latin9]{inputenc}
\usepackage{float}
\usepackage{amsmath}
\usepackage{amsthm}
\usepackage{amssymb}

\makeatletter

\providecommand{\tabularnewline}{\\}

\theoremstyle{plain}
\newtheorem{thm}{\protect\theoremname}

\usepackage{xcolor}
\usepackage{natbib}
\usepackage[unicode=true,bookmarks=true,breaklinks=true,
bookmarksnumbered=false,bookmarksopen=false,pdfborder={0 0 0},
pdfborderstyle={},backref=false,colorlinks=true,citecolor=cyan,
linkcolor=cyan]{hyperref}
\usepackage[titletoc]{appendix}
\usepackage{bm}
\usepackage{autonum}
\usepackage[final]{microtype}
\usepackage{lmodern}
\makeatletter
\def\ps@pprintTitle{%
 \let\@oddhead\@empty
 \let\@evenhead\@empty
 \def\@oddfoot{}%
 \let\@evenfoot\@oddfoot}
\makeatother

\makeatother

\usepackage{babel}
\providecommand{\theoremname}{Theorem}

\begin{document}

\begin{frontmatter}{}

\title{Orthogonal polynomials and Möbius transformations}

\author{R. S. Vieira}

\ead{rs.vieira@unesp.br}

\author{V. Botta}

\ead{vanessa.botta@unesp.br}

\address{Universidade Estadual Paulista -- UNESP, Faculdade de Ciências e
Tecnologia, Departamento de Matemática e Computação, CEP. 19060-900,
Presidente Prudente, SP, Brasil.}
\begin{abstract}
Given an orthogonal polynomial sequence on the real line, another
sequence of polynomials can be found by composing these polynomials
with a general Möbius transformation. In this work, we study the properties
of such Möbius-transformed polynomials. We show that they satisfy
an orthogonality relation in given curve of the complex plane with
respect to a varying weight function and that they also enjoy several
properties common to the orthogonal polynomial sequences on the real
line --- e.g. a three-term recurrence relation, Christoffel-Darboux
type identities, their zeros are simple, lie on the support of orthogonality
and have the interlacing property, etc. Moreover, we also show that
the Möbius-transformed polynomials obtained from classical orthogonal
polynomials also satisfy a second-order differential equation, a Rodrigues'
type formula and generating functions. As an application, we show
that Hermite, Laguerre, Jacobi, Bessel and Romanovski polynomials
are all related to each other by a suitable Möbius transformation.
New orthogonality relations for Bessel and Romanovski polynomials
are also presented. 
\end{abstract}

\end{frontmatter}{}

\section{Introduction}

Since its introduction in mathematical sciences, the concept of orthogonal
polynomials has been extended and generalized in many ways. The first
concept of polynomial orthogonality referred to sequences of polynomials
orthogonal with respect to a real, non-negative and continuous weight
function (the classical orthogonal polynomials of Hermite, Laguerre
and Jacobi are the most known examples). This concept was soon after
generalized in order to cover the cases of discrete measures (the
Charlier polynomials being the prototype in this class). The range
of known families of orthogonal polynomials increased exponentially
after Askey showed that hypergeometric functions lead to several families
of orthogonal polynomials; together with their $q$-analogues, a concept
that was introduced by Hahn, this resulted in \emph{Askey scheme}
for the classification of hypergeometric orthogonal polynomials \citep{KoekoekEtal2010}.
All those polynomials are orthogonal on the real line; of course,
an obvious generalization is to consider polynomials that are orthogonal
in other curves of the complex plane. This was already taken into
account, the most known example being the orthogonal polynomial sequences
on the unit circle \citep{Simon2005A,Simon2005B}. Since then, many
other concepts of polynomial orthogonality were defined: we can cite,
for instance, quasi-orthogonal polynomials \citep{Chihara1957}, polynomials
that are orthogonal with respect to a varying measure \citep{Totik1998},
exceptional orthogonal polynomials \citep{Gomez2009,Gomez2010}, para-orthogonal
polynomials \citep{Jones1989}, among others. 

In this work, we present another generalized class of orthogonal polynomials
which we call \emph{Möbius-transformed orthogonal polynomials}. As
the name says, these sequences of polynomials are generated by the
action of a general Möbius transformation over an orthogonal polynomial
sequence on the real line. In  what follows, we shall show that such
sequences of polynomials are orthogonal in a given curve of the complex
plane, with respect to a simple varying weight function. Besides,
these Möbius-transformed orthogonal polynomials enjoy many properties
of the usual orthogonal polynomials on the real line as, for example,
their zeros are all simple, lie on the support of orthogonality and
have the interlacing property; they also satisfy a three-term recurrence
relation and identities similar to those of Christoffel-Darboux can
be derived. In particular, the Möbius-transformed polynomials obtained
from a classical orthogonal polynomial sequence also satisfy second-order
differential equations, their weight function also satisfy a first-order
differential equation of Pearson's type, there also exist Rodrigues'
type formula for them and they can also be found through simple generating
functions.

We highlight that these Möbius-transformed orthogonal polynomials
can be of importance in Applied Mathematics and Physics. In fact,
the composition of polynomials with a Möbius transformation is already
considered in many applications: we can cite, for example, its role
in some of the fastest polynomial isolating and root-finding methods
available to date (e.g., Akritas method \citep{Akritas1989}). We
can also find the number of zeros that a complex polynomial has on
the unit circle through Sturm theorem if we previously consider the
composition of the testing polynomial with the following pair of Möbius
transformations (often called Cayley transformations), 
\begin{equation}
M(x)=-i\frac{x+1}{x-1},\qquad W(x)=\frac{x-i}{x+i},\label{Cayley}
\end{equation}
that map, respectively, the real line onto the unit circle and vice
versa \citep{Vieira2019}. Moreover, if we act with the first of these
special Möbius transformations over an orthogonal polynomial sequence
on the real line, then we can verify (from the results that will be
reported in this work) that a new sequence of polynomials whose zeros
lie all on the unit circle is obtained. Consequently, we shall get
an infinite sequence of \emph{self-inversive} or \emph{self-reciprocal}
polynomials (which are complex or real polynomials whose zeros are
all symmetric with respect to the unit circle). It is worth to mention
that self-inversive and self-reciprocal polynomials are important
in many fields of Physics and Mathematics --- see \citep{Vieira2019B}
and references therein. Finally, such Möbius-transformed polynomials
are also orthogonal on the unit circle with respect to a varying measure;
because of this, we can conclude that they are intimately related
with the usual orthogonal polynomials on the unit circle and also
with para-orthogonal polynomials. The study of this special case will
be reported a separated paper due to particular importance (in this
work we shall consider a general Möbius transformation). 

Another field where Möbius-transformed polynomials can be of relevance
is related to boundary value problems and second-order differential
equations, which are usually found in physical problems. It is well
known the role of classical orthogonal polynomials in Sturm-Liouville
problems on the real line; because Möbius transformations preserve
the main characteristics of the polynomials, we can conclude from
this that the corresponding Möbius-transformed polynomials will be
associated with boundary values problems defined in some curves (arc
circles or line segments) of the complex plane and that they will
also obey second-order differential equations with complex coefficients.
We can cite the so-called \emph{relativistic classical orthogonal
polynomials}, which are orthogonal polynomials on the real line with
respect to a varying measure that appear in some problems of relativistic
quantum mechanics \citep{Aldaya1991,Natalini1996,HeNatalini1999}.
These relativistic classical orthogonal polynomials are actually special
cases of Möbius-transformed orthogonal polynomials --- in fact, it
was showed by Ismail that they can be obtained from Jacobi polynomials
through the action of a specific Möbius transformation with real parameters
\citep{Ismail1996}.

This work is organized as follows: the concept of Möbius-transformed
polynomials is introduced in Section \ref{Section Mobius}. In Section
\ref{MTOP}, we begin our study on Möbius-transformed orthogonal polynomials,
where their orthogonality, the three-term recurrence relations and
the corresponding Christoffel-Darboux type identities are discussed.
The classical Möbius-transformed polynomials are studied in Section
\ref{Section CMTOP}, where we show that they also satisfy second-order
differential equations, that their weight functions satisfy a first-order
Pearson's type differential equation, we derive Rodrigues' like formulas
for them and the respective generating functions. Finally, in Section
\ref{Section Applications} we show, as an application, that all the
sequences of Hermite, Laguerre, Jacobi, Bessel and Romanovski polynomials
are related to each other essentially by a Möbius transformation;
this also provided new orthogonality relations for Bessel and Romanovski
polynomials. 

\section{Möbius-transformed polynomials\label{Section Mobius}}

A Möbius transformation is a conformal mapping defined of the extended
complex plane $\mathbb{C}_{\infty}=\mathbb{C}\cup\{\infty\}$ (i.e.,
the Riemann sphere) by the formula \citep{Flanigan1972,Ahlfors1979}:
\begin{equation}
M(x)=\frac{ax+b}{cx+d},\label{Mx}
\end{equation}
where $x\in\mathbb{C}_{\infty}$ and $a$, $b$, $c$ and $d$ are
complex numbers such that $\varDelta=ad-bc\neq0$. Its inverse, which
is also a Möbius transformation, is given by the expression, 
\begin{equation}
W(x)=-\frac{dx-b}{cx-a}.\label{Wx}
\end{equation}
Together with the relations 
\begin{align}
M(-d/c) & =\infty, & M(\infty) & =a/c, & W(a/c) & =\infty, & W(\infty) & =-d/c,
\end{align}
transformations (\ref{Mx}) and (\ref{Wx}) become continuous on $\mathbb{C}_{\infty}$.
The Möbius transformations include as special cases many important
geometric transformations as, for instance, translations $M(x)=x+b$,
rotations $M(x)=ax$ $(|a|=1)$, inversions $M(x)=1/x$ and linear
transformations $M(x)=ax+b$. They also have many interesting properties
which can be found in \citep{Flanigan1972,Ahlfors1979}.

Given a complex polynomial $p(x)=p_{0}+p_{1}x+\cdots+p_{n-1}x^{n-1}+p_{n}x^{n}$
of degree $n$, we can compose $p(x)$ with $M(x)$ so that rational
function 
\begin{equation}
r(x)=p_{0}+p_{1}\left(\frac{ax+b}{cx+d}\right)+\cdots+p_{n-1}\left(\frac{ax+b}{cx+d}\right)^{n-1}+p_{n}\left(\frac{ax+b}{cx+d}\right)^{n},\label{r}
\end{equation}
is is obtained as the result. Thus, multiplying $r(x)$ by $\left(cx+d\right)^{n}$,
a new polynomial is obtained: 
\begin{multline}
q(x)=\left(cx+d\right)^{n}r(x)=\left(cx+d\right)^{n}p(M(x))\\
=p_{0}\left(cx+d\right)^{n}+p_{1}\left(ax+b\right)\left(cx+d\right)^{n-1}+\cdots+p_{n-1}\left(ax+b\right)^{n-1}\left(cx+d\right)+p_{n}\left(ax+b\right)^{n}.\label{q}
\end{multline}
We shall refer to $r(x)$ as a \emph{Möbius-transformed rational function}
and to $q(z)$ as a \emph{Möbius-transformed polynomial}. Notice that
the original polynomial $p(x)$ can be retrieved from $r(x)$ and
$q(x)$ by the formula
\begin{equation}
p(x)=r(W(x))=\frac{q(W(x))}{\left(cW(x)+d\right)^{n}}=\frac{\left(-1\right)^{n}}{\varDelta^{n}}\left(cx-a\right)^{n}q(W(x)).\label{P}
\end{equation}

Not always the Möbius-transformed polynomial $q(x)$, as defined in
(\ref{q}), has the same degree as the original polynomial $p(x)$.
The precise condition for this to be true is the following:
\begin{thm}
\label{PropNdegree}Let $p(x)=p_{n}x^{n}+\cdots+p_{0}$ be a complex
polynomial of degree $n$. Then, the Möbius-transformed polynomial
$q(x)$ will be polynomial of degree $n$ if, and only if, $c=0$
or, either, if $c\neq0$ and $p(a/c)\neq0$.
\end{thm}
\begin{proof}
If $c=0$ the proof is trivial; thereby, assume $c\neq0$. Expanding
the binomials in (\ref{Mx}) we find that the leading coefficient
$q_{n}$ of $q(x)$ is given by the expression, 
\begin{equation}
q_{n}=\sum_{k=0}^{n}p_{n-k}a^{n-k}c^{k}=c^{n}\sum_{k=0}^{n}p_{n-k}\left(\frac{a}{c}\right)^{n-k}=c^{n}\left[p_{n}\left(\frac{a}{c}\right)^{n}+\cdots+p_{0}\right]=c^{n}p\left(\frac{a}{c}\right),\label{qn}
\end{equation}
from which we conclude that $q(x)$ will be degree $n$ whenever $a/c$
is not a zero of $p(x)$. 
\end{proof}
We can also find how the zeros of the Möbius-transformed polynomial
$q(x)$ are related with the zeros of the original polynomial $p(x)$
through the following:
\begin{thm}
\label{PropRoots} Let $\xi_{1},\ldots,\xi_{n}$ denote the zeros
of a polynomial $p(x)$ of degree $n$ and $\zeta_{1},\ldots,\zeta_{n}$
the respective zeros of the Möbius-transformed polynomial $q(x)$.
Assuming that $p(a/c)\neq0$, we have that, 
\begin{equation}
\zeta_{1}=W\left(\xi_{1}\right),\ldots,\zeta_{n}=W(\xi_{n}).
\end{equation}
\end{thm}
\begin{proof}
For any zero $\xi_{k}$ $\left(1\leqslant k\leqslant n\right)$ of
$P_{n}(x)$ we have from (\ref{P}) that: 
\begin{equation}
p\left(\xi_{k}\right)=\frac{\left(-1\right)^{n}}{\varDelta^{n}}\left(c\xi_{k}-a\right)^{n}q\left(W\left(\xi_{k}\right)\right)=0,\qquad1\leqslant k\leqslant n.
\end{equation}
This plainly shows us that $\zeta_{k}=W(\xi_{k})$ will be a zero
of $Q_{n}(x)$ whenever $\xi_{k}\neq a/c$.
\end{proof}
From this we can also conclude that the Möbius-transformed polynomial
$q(x)$ will be a polynomial of degree $n-m$ whenever the original
polynomial $p(x)$ of degree $n$ has a zero of multiplicity $m$
at the point $x=a/c$. This, of course, is due to the cancellation
of the $m$ poles generated by the mapping $W\left(a/c\right)=\infty$
with the corresponding $m$ zeros of the factors $\left(c\xi-a\right)$
generated by the action of the Möbius transformation over $p(z)$.

\section{Möbius-transformed orthogonal polynomials\label{MTOP}}

Now, let us consider an orthogonal polynomial sequence defined in
an given interval $(l,r)\subseteq\mathbb{R}$ with respect to a given
weight function $w(x)$. That is, let us consider a sequence $P=\left\{ P_{n}(x)\right\} _{n=0}^{\infty}$
in which each polynomial $P_{n}(x)\in P$ has degree $n$ and such
that the following \emph{orthogonality relation} is satisfied \citep{Szego1939,Chihara2011}:
\begin{equation}
\int_{l}^{r}P_{m}(x)P_{n}(x)w(x)\mathrm{d}x=K_{m}\delta_{m,n},\label{OCP}
\end{equation}
where $\delta_{m,n}$ is the Kronecker symbol, $K_{m}$ are positive
constants and the interval $(l,r)\in\mathbb{R}$ can be either finite,
semi-infinite or infinite\footnote{We remark that definition above can also be generalized to the case
where the weight function $w(x)$ is discrete by interpreting the
integral in the sense of Stieltjes or Lebesgue, although in this work
we shall only consider the case where $w(x)$ is continuous. }. 

Composing each polynomial $P_{n}(x)\in P$ with the Möbius transformation
(\ref{Mx}), we shall obtain a sequence $R=\left\{ R_{n}(x)\right\} _{n=0}^{\infty}$
of rational functions, where, 
\begin{equation}
R_{n}(x)=P_{n}\left(M(x)\right),\qquad n\geqslant0.\label{Rn}
\end{equation}
Multiplying each rational function $R_{n}(x)$ by $\left(cx+d\right)^{n}$,
a sequence of Möbius-transformed polynomials $Q=\left\{ Q_{n}(x)\right\} _{n=0}^{\infty}$
is obtained: 
\begin{equation}
Q_{n}(x)=\left(cx+d\right)^{n}R_{n}(x)=\left(cx+d\right)^{n}P_{n}(M(x)),\qquad n\geqslant0.\label{Qn}
\end{equation}
As before, the relationship between $P_{n}(x)$, $Q_{n}(x)$ and $R_{n}(x)$
is the following: 
\begin{equation}
P_{n}(x)=R_{n}(W(x))=\frac{Q_{n}(W(x))}{\left(cW(x)+d\right)^{n}}=\frac{\left(-1\right)^{n}}{\varDelta^{n}}\left(cx-a\right)^{n}Q_{n}(W(x)),\qquad n\geqslant0.\label{Pn}
\end{equation}

The properties of such sequences of Möbius-transformed rational functions
and polynomials will be analyzed in what follows.

\subsection{Orthogonality relations\label{Subsection OR}}

Let us begin our analysis by presenting the orthogonality relation
satisfied by the Möbius-transformed rational functions $R_{n}(x)$
and polynomials $Q_{n}(x)$.
\begin{thm}
\label{PropOrtho}Let $P=\left\{ P_{n}(x)\right\} _{n=0}^{\infty}$
be an orthogonal polynomial sequence on the real line whose polynomials
$P_{n}(x)$ satisfy the orthogonality relation (\ref{OCP}) for a
given weight function $w(x)$ which does not vanish along the orthogonality
interval $I=\left(l,r\right)\subseteq\mathbb{R}$. Then, the sequence
$R=\left\{ R_{n}(x)\right\} _{n=0}^{\infty}$ of the Möbius-transformed
rational functions $R_{n}(x)$ defined in (\ref{Rn}) will satisfy
the following orthogonality relations: 
\begin{equation}
\int_{\varGamma}R_{m}(x)R_{n}(x)\omega(x)\mathrm{d}x=K_{m}\delta_{m,n},\label{OCR}
\end{equation}
where 
\begin{equation}
\omega(x)=w(M(x))M'(x)=\varDelta\frac{w(M(x))}{\left(cx+d\right)^{2}}.\label{omega}
\end{equation}
Similarly, the sequence $Q=\left\{ Q_{n}(x)\right\} _{n=0}^{\infty}$
of Möbius-transformed polynomials $Q_{n}(x)$ defined in (\ref{Qn})
will satisfy orthogonality relations of the form, 
\begin{equation}
\int_{\varGamma}Q_{m}(x)Q_{n}(x)\omega_{m,n}(x)\mathrm{d}x=K_{m}\delta_{m,n}\qquad\text{where}\qquad\omega_{m,n}(x)=\frac{\omega(x)}{\left(cx+d\right)^{m+n}}.\label{OCQ}
\end{equation}
In both cases, the path of integration is 
\begin{equation}
\varGamma=\{z\in\mathbb{C}:z=W(x),\lambda<z<\rho\},\qquad\text{where}\qquad\lambda=W(l)\qquad\text{and}\qquad\rho=W(r),\label{Gamma}
\end{equation}
which corresponds to a curve (a straight line or a circle) in the
complex plane whose initial and final points are $\lambda=W(l)$ and
$\rho=W(r)$, respectively. Finally, the new weight functions $\omega(x)$
and $\omega_{m,n}(x)$ do not vanish along the curve $\varGamma$,
except at infinity.
\end{thm}
\begin{proof}
The orthogonality relation (\ref{OCR}) for the Möbius-transformed
rational functions $R_{n}(x)$ follows from the change of variable
$x=M(y)$ in the integral (\ref{OCP}), taking into account that the
Jacobian is $M'(x)=\varDelta/\left(cx+d\right)^{2}$. The corresponding
orthogonality condition (\ref{OCQ}) for the Möbius-transformed polynomials
$Q_{n}(x)$ follows from their definitions given in (\ref{Qn}) and
from (\ref{OCR}), after we absorb the factors $\left(cx+d\right)^{m}$
and $\left(cx+d\right)^{n}$ into the varying weight function $\omega_{m,n}(x)$.
Besides, notice that the path of integration $I=\{x\in\mathbb{R}:l<x<r\}$
is mapped through the inverse Möbius-transformation, $W(x)$, onto
the curve $\varGamma$, as defined in (\ref{Gamma}), as well as the
limits $x=l$ and $x=r$ are respectively mapped to $\lambda=W(l)$
and $\rho=W(r)$. Finally, consider a point $y$ on the curve $\varGamma$
so that we can write $y=W(x)$ where $x\in I$; in this case, (\ref{omega})
simplifies to $\omega(W(x))=-w(x)\left(cx-a\right)$. Thus, provided
that $w(x)$ does not vanish on $I$, we conclude that $\omega(W(x))$
can vanish only at $x=a/c$; but this means that $y=W(a/c)=\infty$
so that $\omega(y)$ can vanish only at $y=\infty$ along the curve
$\varGamma$, and the same line of reasoning holds for $\omega_{m,n}(x)$. 
\end{proof}
Notice that integrals (\ref{OCR}) and (\ref{OCQ}) should be regarded
as a complex contour integrals. Whenever the point $x=a/c$ is not
on the orthogonality curve $I=\left(l,r\right)\subseteq\mathbb{R}$,
there is nothing to worry about, as in this case the orthogonality
curve $\varGamma$ will be finite, open and connected, so that these
integrals will depend only on their extreme points $\lambda=W(l)$
and $\rho=W(r)$. However, some care should be taken otherwise. In
fact, if $x=a/c$ is on $I$, then the curve $\varGamma$ will no
longer be connected: instead, it will be constituted by two disjoint
straight lines in the complex plane passing through the point $W(a/c)=\infty$
--- the first one starts at $\lambda=W(l)$ and goes towards the
complex infinity, while the other starts at infinity and ends at $\rho=W(r)$.
Notice as well that if $I$ comprehends the whole real line (as in
the case of Hermite polynomials, for instance), then $\varGamma$
become either an infinite straight line or a circle in the complex
plane; in fact, we have that $W(-\infty)=W(\infty)=-d/c$. Notice
that in the extended complex plane $\mathbb{C}_{\infty}$, $\varGamma$
will be always a connected curve.

The weight function $\omega_{m,n}(x)$ that appears in the orthogonality
relation (\ref{OCQ}) is a special type of which is called a \emph{varying
weight function}. Varying weight functions are measures that depend
on some parameters of the polynomials entering in the orthogonality
relation. Such kind of weights arises naturally from sequences of
orthogonal polynomials $P=\left\{ P_{n}^{\alpha}(x)\right\} _{n=0}^{\infty}$
whose members depend on one (or more) parameter $\alpha$ (e.g., the
associated Laguerre polynomials $L_{n}^{\alpha}(x)$). In fact, provided
that, for fixed $\alpha$, these polynomials satisfy an orthogonality
relation of the usual form (\ref{OCP}) with respect with a given
weight function $w_{\alpha}(x)$, we can construct from $P$ a modified
sequence $P^{\left\{ \alpha\right\} }=\left\{ P_{n}^{\alpha_{n}}(x)\right\} _{n=0}^{\infty}$
in which each polynomial $P_{n}^{\alpha_{n}}(x)$ has a different
parameter $\alpha_{n}$; it follows in this case that the polynomials
$P_{n}^{\alpha_{n}}(x)$ will satisfy an orthogonality relation with
respect to a varying weight function, namely:
\[
\int_{l}^{r}P_{m}^{\alpha_{m}}(x)P_{n}^{\alpha_{n}}(x)w_{\alpha_{m},\alpha_{n}}(x)\mathrm{d}x=K_{m}\delta_{m,n}\qquad\text{where},\qquad w_{\alpha_{m},\alpha_{n}}(x)=\begin{cases}
w_{\alpha_{n}}(x), & m<n,\\
w_{\alpha_{m}}(x), & m>n,\\
w_{\alpha}(x), & m=n,
\end{cases}
\]
for some positive constants $K_{m}$, $m\geqslant0$. This works because
every polynomial of degree $n$ that belongs to a sequence of orthogonal
polynomials on the real line is orthogonal to any polynomial of degree
less than $n$ --- thus, to establish the orthogonality of a sequence
of polynomials with different parameters we just need to pick up the
weight function belonging to the polynomial of higher degree appearing
on the orthogonality relation; besides, when the two polynomials have
the same degree, they will also have the same value of $\alpha$ so
that the value of the integral is the same as that one for the fixed
value $\alpha_{m}=\alpha_{n}$ of $\alpha$. Finally, we remark that
sequences of polynomials orthogonal with respect to varying measures
appear in several applications, both in mathematics \citep{Totik1998}
as in physics \citep{Dehesa2001}. 

\subsection{Three-term recurrence relations}

It is well-known that any sequence $P=\left\{ P_{n}(x)\right\} _{n=0}^{\infty}$
of orthogonal polynomials on the real line satisfies a three-term
recurrence relation of the form \citep{Szego1939,Chihara2011}: 
\begin{equation}
P_{n}(x)=\left(A_{n}x-B_{n}\right)P_{n-1}(x)-C_{n}P_{n-2}(x),\qquad n\geqslant1,\label{RRP}
\end{equation}
with the initial conditions $P_{-1}(x)=0$ and $P_{0}(x)=1$, where
$A_{n}$, $B_{n}$ and $C_{n}$ are real constants with $A_{n}>0$
and $C_{n}>0$ for each $n\geqslant1$. In the following, we shall
show that the sequences $R=\left\{ R_{n}(x)\right\} _{n=0}^{\infty}$
and $Q=\left\{ Q_{n}(x)\right\} _{n=0}^{\infty}$ of Möbius-transformed
rational functions and polynomials also satisfy three-term recurrence
relation.
\begin{thm}
\label{PropRecursion}Let $P=\left\{ P_{n}(x)\right\} _{n=0}^{\infty}$
be orthogonal polynomial sequence on the real line whose polynomials
$P_{n}(x)$ satisfy the three-term recurrence condition (\ref{RRP}).
Then, the sequence $R=\left\{ R_{n}(x)\right\} _{n=0}^{\infty}$ of
rational functions $R_{n}(x)$ defined in (\ref{Rn}) will satisfy
a three term relation of the following form:
\begin{equation}
R_{n}(x)=\left[A_{n}M(x)-B_{n}\right]R_{n-1}(x)-C_{n}R_{n-2}(x)=\left[A_{n}\left(\frac{ax+b}{cx+d}\right)-B_{n}\right]R_{n-1}(x)-C_{n}R_{n-2}(x),\qquad n\geqslant1,\label{RRR}
\end{equation}
with the initial conditions $R_{-1}(x)=0$ and $R_{1}(x)=1$. Similarly,
the sequence $Q=\left\{ Q_{n}(x)\right\} _{n=0}^{\infty}$ of Möbius-transformed
polynomials $Q_{n}(x)$ defined in (\ref{Qn}), will satisfy the three-term
recurrence relation: 
\begin{equation}
Q_{n}(x)=\left(\mathcal{A}_{n}x-\mathcal{B}_{n}\right)Q_{n-1}(x)-\mathcal{C}_{n}\left(cx+d\right)^{2}Q_{n-2}(x)\qquad n\geqslant1,\label{RRQ}
\end{equation}
with the initial conditions $Q_{-1}(x)=0$ and $Q_{0}(x)=1$, where,
\begin{equation}
\mathcal{A}_{n}=\left(aA_{n}-cB_{n}\right),\qquad\mathcal{B}_{n}=\left(dB_{n}-bA_{n}\right),\qquad\mathcal{C}_{n}=C_{n}.
\end{equation}
\end{thm}
\begin{proof}
Inserting relation (\ref{Pn}) into (\ref{RRP}), we get that, 
\begin{equation}
R_{n}(W(x))=\left(A_{n}x-B_{n}\right)R_{n-1}(W(x))-C_{n}R_{n-2}(W(x)).
\end{equation}
Now, making the change of variable $x=M(y)$, we obtain, at once,
\begin{align}
R_{n}(y) & =\left(A_{n}M(y)-B_{n}\right)R_{n-1}(y)-C_{n}R_{n-2}(y)=\left[A_{n}\left(\frac{ay+b}{cy+d}\right)-B_{n}\right]R_{n-1}(y)-C_{n}R_{n-2}(y).
\end{align}
It is clear that the relation above holds for any $n\geqslant1$ provided
we define $R_{-1}(x)=0$ and $R_{1}(x)=1$. Finally, (\ref{RRQ})
follows after we insert (\ref{Qn}) into the above expression. In
fact, we get in this way, 
\begin{equation}
\frac{Q_{n}(y)}{\left(cy+d\right)^{n}}=\left[A_{n}\left(\frac{ay+b}{cy+d}\right)-B_{n}\right]\frac{Q_{n-1}(y)}{\left(cy+d\right)^{n-1}}-C_{n}\frac{Q_{n-2}(y)}{\left(cy+d\right)^{n-2}},
\end{equation}
so that,
\begin{align}
Q_{n}(y) & =\left[A_{n}\left(\frac{ay+b}{cy+d}\right)-B_{n}\right]\left(cy+d\right)Q_{n-1}(y)-C_{n}\left(cy+d\right)^{2}Q_{n-2}(y),\nonumber \\
 & =\left[A_{n}\left(ay+b\right)-B_{n}\left(cy+d\right)\right]Q_{n-1}(y)-C_{n}\left(cy+d\right)^{2}Q_{n-2}(y),\nonumber \\
 & =\left[\left(aA_{n}-cB_{n}\right)y-\left(dB_{n}-bA_{n}\right)\right]Q_{n-1}(y)-C_{n}\left(cy+d\right)^{2}Q_{n-2}(y),
\end{align}
 which holds for any $n\geqslant1$ provided we define $Q_{-1}(x)=0$
and $Q_{0}(x)=1$.
\end{proof}

\subsection{Christoffel-Darboux identities }

Another important relation satisfied by any orthogonal polynomial
sequence on the real line is the so-called \emph{Christoffel-Darboux
identity} \citep{Szego1939,Chihara2011}: 
\begin{equation}
\sum_{k=0}^{n}\frac{A_{k+1}}{C_{0}\cdots C_{k+1}}P_{k}(x)P_{k}(y)=\frac{1}{C_{0}\cdots C_{n+1}}\frac{P_{n+1}(x)P_{n}(y)-P_{n}(x)P_{n+1}(y)}{x-y},\label{CDP}
\end{equation}
and its confluent form:
\begin{equation}
\sum_{k=0}^{n}\frac{A_{k+1}}{C_{0}\cdots C_{k+1}}P_{k}^{2}(x)=\frac{P_{n+1}'(x)P_{n}(x)-P_{n}'(x)P_{n+1}(x)}{C_{0}\cdots C_{n+1}}.\label{CCDP}
\end{equation}
 Similar identities also exist for the sequences of Möbius-transformed
rational functions and polynomials, as we shall describe in the following:
\begin{thm}
\label{PropCD}Let $P=\left\{ P_{n}(x)\right\} _{n=0}^{\infty}$ be
an orthogonal polynomial sequence on the real line that satisfies
the Christoffel-Darboux identities (\ref{CDP}) and (\ref{CCDP}).
Then the sequence $R=\left\{ R_{n}(x)\right\} _{n=0}^{\infty}$ of
Möbius-transformed rational functions $R_{n}(x)$ defined in (\ref{Rn})
will satisfy the following identities: 
\begin{equation}
\sum_{k=0}^{n}\frac{A_{k+1}}{C_{0}\cdots C_{k+1}}R_{k}(x)R_{k}(y)=\left(cx+d\right)\left(cy+d\right)\left[\frac{R_{n+1}(x)R_{n}(y)-R_{n}(x)R_{n+1}(y)}{\varDelta\left(C_{0}\cdots C_{n+1}\right)\left(x-y\right)}\right],\label{CDR}
\end{equation}
and 
\begin{equation}
\sum_{k=0}^{n}\frac{A_{k+1}}{C_{0}\cdots C_{k+1}}R_{k}^{2}(y)=\left(cy+d\right)^{2}\left[\frac{R_{n+1}'(y)R_{n}(y)-R_{n}'(y)R_{n+1}(y)}{\varDelta\left(C_{0}\cdots C_{n+1}\right)}\right].\label{CCDR}
\end{equation}
Similarly, the sequence $Q=\left\{ Q_{n}(x)\right\} _{n=0}^{\infty}$
of Möbius-transformed polynomials $Q_{n}(x)$ defined in (\ref{Qn})
will satisfy the identities: 
\begin{equation}
\sum_{k=0}^{n}\frac{A_{k+1}}{C_{0}\cdots C_{k+1}}\frac{Q_{k}(x)}{\left(cx+d\right)^{k}}\frac{Q_{k}(y)}{\left(cy+d\right)^{k}}=\frac{Q_{n+1}(x)Q_{n}(y)\left(cy+d\right)-\left(cx+d\right)Q_{n}(x)Q_{n+1}(y)}{\varDelta\left(C_{0}\cdots C_{n+1}\right)\left(cx+d\right)^{n}\left(cy+d\right)^{n}\left(x-y\right)},\label{CDQ}
\end{equation}
and 
\begin{equation}
\sum_{k=0}^{n}\frac{A_{k+1}}{C_{0}\cdots C_{k+1}}\frac{Q_{k}^{2}(x)}{\left(cx+d\right)^{2k}}=\frac{Q_{n+1}'(x)Q_{n}(x)-Q_{n}'(x)Q_{n+1}(x)-cQ_{n+1}(x)Q_{n}(x)}{\varDelta\left(C_{0}\cdots C_{n+1}\right)\left(cx+d\right)^{2n-1}}.\label{CCDQ}
\end{equation}
\end{thm}
\begin{proof}
From (\ref{Pn}) we get that (\ref{CDP}) becomes, 
\begin{equation}
\sum_{k=0}^{n}\frac{A_{k+1}}{C_{0}\cdots C_{k+1}}R_{k}(W(x))R_{k}(W(y))=\frac{1}{C_{0}\cdots C_{n+1}}\frac{R_{n+1}(W(x))R_{n}(W(y))-R_{n}(W(x))R_{n+1}(W(y))}{x-y}.
\end{equation}
 Now, making the change of variables $x=M(z)$ and $t=M(y)$, we obtain,
\begin{equation}
\sum_{k=0}^{n}\frac{A_{k+1}}{C_{0}\cdots C_{k+1}}R_{k}(z)R_{k}(t)=\frac{1}{C_{0}\cdots C_{n+1}}\frac{R_{n+1}(z)R_{n}(t)-R_{n}(z)R_{n+1}(t)}{M(z)-M(t)}.
\end{equation}
 But $M(z)-M(t)=\varDelta(z-t)/\left[\left(ct+d\right)\left(cz+d\right)\right]$,
so that we get, 
\begin{equation}
\sum_{k=0}^{n}\frac{A_{k+1}}{C_{0}\cdots C_{k+1}}R_{k}(z)R_{k}(t)=\frac{1}{\varDelta}\frac{\left(cz+d\right)\left(ct+d\right)}{C_{0}\cdots C_{n+1}}\left[\frac{R_{n+1}(z)R_{n}(t)-R_{n}(z)R_{n+1}(t)}{z-t}\right],
\end{equation}
 which is of the same form as (\ref{CDR}). Now, to prove (\ref{CCDR}),
we proceed in a similar fashion: using (\ref{Pn}) into (\ref{CCDP}),
we get, 
\begin{equation}
\sum_{k=0}^{n}\frac{A_{k+1}}{C_{0}\cdots C_{k+1}}R_{k}^{2}(W(x))=\frac{1}{C_{0}\cdots C_{n+1}}\left\{ \frac{\mathrm{d}}{\mathrm{d}x}\left[R_{n+1}(W(x))\right]R_{n}(W(x))-\frac{\mathrm{d}}{\mathrm{d}x}\left[R_{n}(W(x))\right]R_{n+1}(W(x))\right\} ,
\end{equation}
that is, 
\begin{equation}
\sum_{k=0}^{n}\frac{A_{k+1}}{C_{0}\cdots C_{k+1}}R_{k}^{2}(W(x))=\frac{W'(x)}{C_{0}\cdots C_{n+1}}\left\{ R_{n+1}'(W(x))R_{n}(W(x))-R_{n}'(W(x))R_{n+1}(W(x))\right\} .
\end{equation}
 Thus, making the change of variables $x=M(y)$ and using the identity
$W'(M(y))=\left(cy+d\right)^{2}/\varDelta$, we obtain, 
\begin{equation}
\sum_{k=0}^{n}\frac{A_{k+1}}{C_{0}\cdots C_{k+1}}R_{k}^{2}(y)=\frac{1}{\varDelta}\frac{\left(cy+d\right)^{2}}{C_{0}\cdots C_{n+1}}\left\{ R_{n+1}'(y)R_{n}(y)-R_{n}'(y)R_{n+1}(y)\right\} ,
\end{equation}
which is identical to (\ref{CCDR}). Finally, to prove (\ref{CDQ})
and (\ref{CCDQ}), we just need to use (\ref{Qn}). With that, (\ref{CDQ})
becomes: 
\begin{multline}
\sum_{k=0}^{n}\frac{A_{k+1}}{C_{0}\cdots C_{k+1}}\frac{Q_{k}(z)}{\left(cz+d\right)^{k}}\frac{Q_{k}(z)}{\left(ct+d\right)^{k}}=\\
\frac{1}{\varDelta}\frac{\left(cz+d\right)\left(ct+d\right)}{C_{0}\cdots C_{n+1}}\frac{1}{z-t}\left[\frac{Q_{n+1}(z)}{\left(cz+d\right)^{n+1}}\frac{Q_{n}(z)}{\left(ct+d\right)^{n}}-\frac{Q_{n}(z)}{\left(cz+d\right)^{n}}\frac{Q_{n+1}(z)}{\left(ct+d\right)^{n+1}}\right],
\end{multline}
so that (\ref{CCDR}) follows after a simplification. Now, using (\ref{Qn})
into (\ref{CCDQ}), we get, 
\begin{multline}
\sum_{k=0}^{n}\frac{A_{k+1}}{C_{0}\cdots C_{k+1}}\frac{Q_{k}^{2}(y)}{\left(cy+d\right)^{2k}}=\\
\frac{1}{\varDelta}\frac{\left(cy+d\right)^{2}}{C_{0}\cdots C_{n+1}}\left\{ \frac{\mathrm{d}}{\mathrm{d}x}\left[\frac{Q_{n+1}(y)}{\left(cy+d\right)^{n+1}}\right]\frac{Q_{n}(y)}{\left(cy+d\right)^{n}}-\frac{\mathrm{d}}{\mathrm{d}x}\left[\frac{Q_{n}(y)}{\left(cy+d\right)^{n}}\right]\frac{Q_{n+1}(y)}{\left(cy+d\right)^{n+1}}\right\} ,
\end{multline}
and (\ref{CCDQ}) follows after we expand the derivatives and simplify
the resulting expression.
\end{proof}

\subsection{Properties of the zeros. }

It is well-known that the zeros of polynomials belonging to a sequence
$P=\left\{ P_{n}(x)\right\} _{n=0}^{\infty}$ of orthogonal polynomials
on the real line are all simple and lie in the orthogonality interval
$\mathcal{I}=(l,r)\subseteq\mathbb{R}$; moreover, for $n>1$, the
zeros of $P_{n}(x)$ interlace with the zeros of $P_{n+1}(x)$ \citep{Szego1939,Chihara2011}.
In this section, we shall show that the Möbius-transformed polynomials
defined in (\ref{Qn}) also present similar properties.
\begin{thm}
\label{PropRealSimple}Let $P_{n}(x)$ be a polynomial of degree $n\geqslant1$
that belongs to a sequence $P=\left\{ P_{n}(x)\right\} _{n=0}^{\infty}$
of orthogonal polynomials defined in a given interval $\mathcal{I}=(l,r)\subseteq\mathbb{R}$
of the real line and such that $P_{n}(a/c)\neq0$ for all $n$. Then,
the zeros of the corresponding Möbius-transformed polynomials $Q_{n}(x)$
defined in (\ref{Qn}), for $n\geqslant1$, will all lie on the curve
of orthogonality $\varGamma$, as defined in (\ref{Gamma}). Moreover,
these zeros are all simple.
\end{thm}
\begin{proof}
From Theorems \ref{PropNdegree} and \ref{PropRoots}, it follows
that the zeros $\zeta_{1}^{n},\ldots,\zeta_{n}^{n}$ of the Möbius-transformed
polynomials $Q_{n}(x)$ are related to the zeros $\xi_{1}^{n},\ldots,\xi_{n}^{n}$
of the original polynomials $P_{n}(x)$ by the relations, 
\begin{equation}
\zeta_{1}^{n}=W(\xi_{1}^{n}),\qquad\zeta_{2}^{n}=W(\xi_{2}^{n}),\qquad\cdots\qquad\zeta_{n}^{n}=W(\xi_{n}^{n}),
\end{equation}
provided that $a/c$ is not a zero of $P_{n}(x)$. This shows us that
zeros of $Q_{n}(x)$ will be all finite and will lie on the curve
$\varGamma$ given in (\ref{Gamma}), which corresponds to the image
of the interval $\mathcal{I}=(l,r)\subseteq\mathbb{R}$ under the
action of the inverse Möbius-transformation $W(x)$. Now, these zeros
must be also simple because the Möbius transformation is one-to-one
and onto on the extended complex plane. 
\end{proof}
\begin{thm}
\label{PropInterlace}The zeros of any two consecutive polynomials
$Q_{n}(x)$ and $Q_{n+1}(x)$, for $n\geqslant1$, that belong to
the sequence $Q=\left\{ Q_{n}(x)\right\} _{n=0}^{\infty}$ of Möbius-transformed
orthogonal polynomials $Q_{n}(x)$ defined in (\ref{Qn}) interlace
on the new curve of orthogonality $\varGamma$ given by (\ref{Gamma}).
\end{thm}
\begin{proof}
The Möbius transformations are one-to-one and onto on the extended
complex plane $\mathbb{C}_{\infty}$. In particular, this implies
that the image of the interval $\mathcal{I}=(l,r)\subseteq\mathbb{R}$
under the mapping $y=W(x)$, which is the curve $\varGamma$, is a
not-self-intersecting curve. Thus, the ordination of any set of points
lying in the interval $\mathcal{I}$ is either unaltered or only inverted
by a Möbius transformation. Because the zeros of $P_{n+1}(x)$ interlaces
with the zeros of $P_{n}(x)$ on $I$, we conclude from what was said
that the zeros of the Möbius-transformed polynomial $Q_{n+1}(x)$
will also interlace with the zeros of the polynomial $Q_{n}(x)$ over
$\varGamma$. 
\end{proof}
We remark that the ordination above should be understood as performed
in the extended complex plane $\mathbb{C}_{\infty}$. In the usual
complex plane, the ordination should be interpreted with care because
the curve $\varGamma$ can be disconnected, as already commented.
In fact, this will occur whenever the point $x=a/c$ lies in the interval
$\mathcal{I}=\left(l,r\right)\subseteq\mathbb{R}$ as, in this case,
the curve $\varGamma$ will cross the infinite. Thus, in this case
the curve $\varGamma$ will be constituted by to disconnected branches
in the complex plane and the ordination described above should be
understood as beginning in the point $\lambda=W(l)$, passing through
the infinity and ending in point $\rho=W(r)$ along the curve $\varGamma$. 

\begin{table}[H]
\centering%
\begin{tabular}{ll}
\hline 
Chebyshev polynomials & Inversion-transformed Chebyshev polynomials\tabularnewline
\hline 
$T_{0}(x)=1$ & $\mathcal{T}_{0}(x)=1$\tabularnewline
$T_{1}(x)=x$ & $\mathcal{T}_{1}(x)=1$\tabularnewline
$T_{2}(x)=-1+2x^{2}$ & $\mathcal{T}_{2}(x)=2-x^{2}$\tabularnewline
$T_{3}(x)=-3x+4x^{3}$ & $\mathcal{T}_{3}(x)=4-3x^{2}$\tabularnewline
$T_{4}(x)=1-8x^{2}+8x^{4}$ & $\mathcal{T}_{4}(x)=8-8x^{2}+x^{4}$\tabularnewline
$T_{5}(x)=5x-20x^{3}+16x^{5}$ & $\mathcal{T}_{5}(x)=16-20x^{2}+5x^{4}$\tabularnewline
$T_{6}(x)=-1+18x^{2}-48x^{4}+32x^{6}$ & $\mathcal{T}_{6}(x)=32-48x^{2}+18x^{4}-x^{6}$\tabularnewline
$T_{7}(x)=-7x+56x^{3}-112x^{5}+64x^{7}$ & $\mathcal{T}_{7}(x)=64-112x^{2}+56x^{4}-7x^{6}$\tabularnewline
$T_{8}(x)=1-32x^{2}+160x^{4}-256x^{6}+128x^{8}$ & $\mathcal{T}_{8}(x)=128-256x^{2}+160x^{4}-32x^{6}+x^{8}$\tabularnewline
\hline 
\end{tabular}

\caption{The first Möbius-transformed orthogonal polynomials $\mathcal{T}_{n}(x)=x^{n}T_{n}(M(x))$,
where $T_{n}(x)$ are the Chebyshev polynomials and $M(x)=1/x$ is
the inversion transformation. Notice that the polynomials $T_{n}(x)$
for odd $n$ have a zero at $x=0$, which implies that the Möbius-transformed
polynomials $\mathcal{T}_{n}(x)$ for odd $n$ will have even degree
$n^{\prime}=n-1$. Notwithstanding, these polynomials are orthogonal
on the interval $\mathcal{J}=\left(-\infty,-1\right)\cup\left(1,\infty\right)$
of the real line with respect to the varying weight function $\omega_{m,n}(x)=x^{-m-n-2}/\sqrt{1-x^{-2}}$.}

\label{InversionChebyshev}
\end{table}

A very interesting case occurs when some of the polynomials belonging
to the sequence $P=\left\{ P_{n}(x)\right\} _{n=0}^{\infty}$ of orthogonal
polynomials on the real line have a zero at the point $x=a/c$. Indeed,
according to Theorem \ref{PropNdegree}, if $P_{n}(a/c)=0$ for some
polynomials $P_{n}(x)\in P$ of degree $n$, then the corresponding
Möbius-transformed polynomials $Q_{n}(x)\in Q$ will have degree $n^{\prime}=n-1$.
Thus, $Q$ will be in this case a \emph{defective} sequence of orthogonal
Möbius-transformed polynomials, that is, a sequence containing polynomials
of repeated degrees whereas other polynomials of specific degrees
are absent. This particularity resembles what is called a sequence
of \emph{exceptional orthogonal polynomials} \citep{Gomez2009,Gomez2010},
which is a sequence of orthogonal polynomials on the real line that
does not contain some polynomials of specific degrees. A sequence
$Q$ of defective orthogonal polynomials, however, differ from the
sequences of exceptional orthogonal polynomials in two points: firstly
they usually are orthogonal in curves of the complex plane (not necessarily
on the real line); secondly, even though the sequence $Q$ contain
polynomials of repeated degrees, the orthogonality condition (\ref{OCQ})
is always satisfied, even for the defective polynomials\footnote{Here we remark that the parameters $m$ and $n$ on the varying weight
function $\omega_{m,n}(x)$ do not rely on the degrees of the polynomials
$Q_{m}(x)$ and $Q_{n}(x)$, but only to the order in which they appear
in the sequence $Q=\left\{ Q_{n}(x)\right\} _{n=0}^{\infty}$ of Möbius-transformed
polynomials.}. 

As an example of defective orthogonal polynomials arising from a Möbius
transformation, we can consider any sequence $S$ of orthogonal polynomials
defined in a symmetrical interval of the real line under the action
of the inversion map $M(x)=1/x$, so that we have $a/c=0$ in this
case. In fact, we know that every polynomial $P_{n}(x)\in S$ of odd
degree (say, $n=2m+1$) will have a zero at $x=0$, thence the corresponding
Möbius-transformed polynomials $Q_{2m+1}(x)\in Q$ will have actually
an even degree (say, $n'=2m$), i.e. they will be defective. The sequence
$Q$ of Möbius-transformed orthogonal polynomials will be, therefore,
constituted by polynomials of even degrees only, and in such a way
that two any consecutive polynomials belonging to this sequence will
have the same degree. In Table \ref{InversionChebyshev} we present
the first orthogonal Möbius-transformed polynomials for the Chebyshev
polynomials under the inversion map.

\section{Classical Möbius-transformed orthogonal polynomials\label{Section CMTOP}}

Among all the sequences of orthogonal polynomials on the real line,
the so-called \emph{classical orthogonal polynomial sequences} ---
which include the sequences of Jacobi, Laguerre and Hermite polynomials
--- are of particular interest. In fact, the classical families of
orthogonal polynomials were the first ones to be systematically studied,
which is justified by their great importance in applied mathematics,
physics and engineering. Besides, the classical orthogonal polynomials
have many properties that are not satisfied by any other sequences
of orthogonal polynomials: for example, they satisfy second-order
differential equations, their weight function also satisfy a linear
differential equation, there exist Rodrigues' type relations and generating
functions from which the classical orthogonal polynomials can be derived
and so on \citep{Szego1939,Chihara2011}. In this section, we shall
show that the Möbius-transformed rational functions and polynomials
that arise from a classical orthogonal polynomial sequence on the
real line enjoy similar properties.

\subsection{Second-order differential equation\label{Subsection ODE}}

It is a well-known result that any sequence $P=\{P_{n}(x)\}_{n=0}^{\infty}$
of classical orthogonal polynomial satisfy a second-order differential
equation of the form \citep{Szego1939,Chihara2011}: 
\begin{equation}
f(x)P_{n}''(x)+g(x)P_{n}'(x)+h_{n}(x)P_{n}(x)=0.\label{ODEP}
\end{equation}
Routh \citep{Routh1885}, Bochner \citep{Bochner1929} and Brenke
\citep{Brenke1930}, have shown, independently, that polynomial solutions
of (\ref{ODEP}) exist only when $f(x)$, $g(x)$ and $h_{n}(x)$
are polynomials of degree $2$, $1$ and $0$, respectively, where
$h_{n}(x)$ can be different for each $n$. Besides, they showed that
the only polynomial solutions of (\ref{ODEP}), up to a linear transformation,
are the classical orthogonal polynomial sequences of Jacobi, Laguerre
and Hermite, besides other few sequences of polynomials which nevertheless
are not orthogonal on the real line in the usual sense\footnote{The first of such non-orthogonal sequences is constituted by polynomials
of the form $\pi_{n}(x)=x^{n}+\alpha x^{m}$, where $\alpha$ is a
real parameter and $m$ any integer satisfying $0\leqslant m\leqslant n-1$.
The second sequence is related to the so-called \emph{Bessel polynomials,}
which were consistently studied for the first time in \citep{KrallFrink1949}.
Finally, the third sequence consists of the Romanovski polynomials
\citep{Romanovski1929}. }. 

In the following, we shall derive the differential equations that
are satisfied by the sequences of the Möbius-transformed rational
functions and polynomials when the original sequence of orthogonal
polynomials is of the classical type. 
\begin{thm}
\label{PropDiff}Let $P=\left\{ P_{n}(x)\right\} _{n=0}^{\infty}$
be orthogonal polynomial sequence on the real line in which each polynomial
$P_{n}(x)$ for $n\geqslant0$ satisfy a second order differential
equation of the form (\ref{ODEP}). Then, each Möbius-transformed
rational function $R_{n}(x)$, as defined by (\ref{Rn}), belonging
to the sequence $R=\left\{ R_{n}(x)\right\} _{n=0}^{\infty}$ will
satisfy a second-order differential equation of the form, 
\begin{equation}
F(x)R_{n}''(x)+G(x)R_{n}'(x)+H_{n}R_{n}(x)=0,\label{ODER}
\end{equation}
where, $F(y)$ and $G(y)$ are polynomials independent of $n$ whose
degree are at most $4$ and $3$, respectively, and $H_{n}$ is a
constant which may depend on $n$. Similarly, each Möbius-transformed
polynomial $Q_{n}(x)$, as defined by (\ref{Qn}), belonging to the
sequence $Q=\left\{ Q_{n}(x)\right\} _{n=0}^{\infty}$ will satisfy
a second-order differential equation of the form, 
\begin{equation}
\mathcal{F}(x)Q_{n}''(x)+\mathcal{G}_{n}(x)Q_{n}'(x)+\mathcal{H}_{n}(x)Q_{n}(x)=0,\label{ODEQ}
\end{equation}
where, now, $\mathcal{F}(x)$ has degree at most $4$ and does not
depend on n, while $\mathcal{G}_{n}(x)$ and $\mathcal{H}_{n}(x)$
are polynomials with degree at most $3$ and $2$, respectively, that
in general depend on $n$.
\end{thm}
\begin{proof}
Inserting relation (\ref{Pn}) on the differential equation (\ref{ODEP}),
we get, 
\begin{equation}
f(x)\left[\frac{\mathrm{d}^{2}}{\mathrm{d}x^{2}}R_{n}(W(x))\right]+g(x)\left[\frac{\mathrm{d}}{\mathrm{d}x}R_{n}(W(x))\right]+h_{n}R_{n}(W(x))=0,
\end{equation}
 Computing the derivatives with the help of chain's rule, we get,
\begin{equation}
f(x)\left[\frac{\mathrm{d}^{2}W(x)}{\mathrm{d}x^{2}}\frac{\mathrm{d}R_{n}(W(x))}{\mathrm{d}W(x)}+\left(\frac{\mathrm{d}W(x)}{\mathrm{d}x}\right)^{2}\frac{\mathrm{d}^{2}R_{n}(W(x))}{\mathrm{d}W(x)^{2}}\right]+g(x)\left[\frac{\mathrm{d}W(x)}{\mathrm{d}x}\frac{\mathrm{d}R_{n}(W(x))}{\mathrm{d}W(x)}\right]+h_{n}R_{n}(W(x))=0,
\end{equation}
Now, the change of variables $x=M(y)$ provides,
\begin{multline}
f(M(y))\left[\frac{\mathrm{d}^{2}W(M(y))}{\mathrm{d}M(y)^{2}}\frac{\mathrm{d}R_{n}(y)}{\mathrm{d}y}+\left(\frac{\mathrm{d}W(M(y))}{\mathrm{d}M(y)}\right)^{2}\frac{\mathrm{d}^{2}R_{n}(y)}{\mathrm{d}y^{2}}\right]\\
+g(M(y))\left[\frac{\mathrm{d}W(M(y))}{\mathrm{d}M(y)}\frac{\mathrm{d}R_{n}(y)}{\mathrm{d}y}\right]+h_{n}R_{n}(y)=0.
\end{multline}
But, 
\begin{equation}
\frac{\mathrm{d}W(M(y))}{\mathrm{d}M(y)}=\frac{(cy+d)^{2}}{\varDelta},\qquad\frac{\mathrm{d}^{2}W(M(y))}{\mathrm{d}M(y)^{2}}=\frac{2c(cy+d)^{3}}{\varDelta^{2}},\label{DWM}
\end{equation}
so that we get, after gathering the derivatives of $R_{n}(y)$,
\begin{equation}
\frac{(cy+d)^{4}}{\varDelta^{2}}f(M(y))\frac{\mathrm{d}^{2}R_{n}(y)}{\mathrm{d}y^{2}}+\left[\frac{2c(cy+d)^{3}}{\varDelta^{2}}f(M(y))+\frac{(cy+d)^{2}}{\varDelta}g(M(y))\right]\frac{\mathrm{d}R_{n}(y)}{\mathrm{d}y}+h_{n}R_{n}(y)=0,
\end{equation}
which is of the same form of (\ref{ODER}), provided we define, 
\begin{equation}
F(y)=\frac{(cy+d)^{4}}{\varDelta^{2}}f(M(y)),\qquad G(y)=\frac{2c(cy+d)^{3}}{\varDelta^{2}}f(M(y))+\frac{(cy+d)^{2}}{\varDelta}g(M(y)),\qquad H_{n}=h_{n}.\label{RFGH}
\end{equation}
Now, because $f(x)$ and $g(x)$ are polynomials with degree at most
$2$ and $1$ respectively, it is convenient to introduce the polynomials
\begin{equation}
\varPhi(x)=(cx+d)^{2}f(M(x)),\qquad\text{and}\qquad\varGamma(x)=(cx+d)g(M(x)),\label{PhiGamma}
\end{equation}
which are, as a matter of a fact, the Möbius-transformed polynomials
of $f(x)$ and $g(x)$, respectively. Notice that $\varPhi(x)$ and
$\varGamma(x)$ have the same degree as $f(x)$ and $g(x)$, from
which we conclude that $F(y)$ and $G(y)$ are polynomials with degree
at most $4$ and $3$, respectively. 

Let us now prove (\ref{ODEQ}). For this we can proceed as follows:
by inserting (\ref{Qn}) into (\ref{ODER}) we get, 
\begin{equation}
F(x)\frac{\mathrm{d}^{2}}{\mathrm{d}x^{2}}\left[\frac{Q_{n}(x)}{\left(cx+d\right)^{n}}\right]+G(x)\frac{\mathrm{d}}{\mathrm{d}x}\left[\frac{Q_{n}(x)}{\left(cx+d\right)^{n}}\right]+H_{n}\left[\frac{Q_{n}(x)}{\left(cx+d\right)^{n}}\right]=0.
\end{equation}
 Computing the derivatives and simplifying, this becomes,
\begin{equation}
F(x)\left[\frac{c^{2}n(n+1)}{\left(cx+d\right)^{2}}Q_{n}(x)-\frac{2cn}{cx+d}Q_{n}'(x)+Q_{n}''(x)\right]+G(x)\left[Q_{n}'(x)-\frac{cn}{cx+d}Q_{n}(x)\right]+H_{n}Q_{n}(x)=0,
\end{equation}
that is, gathering the derivatives of $Q_{n}(x)$, 
\begin{equation}
F(x)Q_{n}''(x)+\left[G(x)-\frac{2cn}{cx+d}F(x)\right]Q_{n}'(x)+\left[\frac{c^{2}n(n+1)}{\left(cx+d\right)^{2}}F(x)-\frac{cn}{cx+d}G(x)+H_{n}\right]Q_{n}(x)=0,
\end{equation}
 which is of the same form than (\ref{ODEQ}), provided we define,
\begin{align}
\mathcal{F}(x) & =F(x)=\frac{(cx+d)^{4}}{\varDelta^{2}}f(M(x)),\label{FF}\\
\mathcal{G}_{n}(x) & =G(x)-\frac{2cn}{cx+d}F(x)=-\frac{2c(n-1)(cx+d)^{3}}{\varDelta^{2}}f(M(x))+\frac{(cx+d)^{2}}{\varDelta}g(M(x)),\label{GG}\\
\mathcal{H}_{n}(x) & =\frac{c^{2}n(n+1)}{\left(cx+d\right)^{2}}F(x)-\frac{cn}{cx+d}G(x)+H_{n}=\frac{c^{2}n(n-1)(cx+d)^{2}}{\varDelta^{2}}f(M(x))-\frac{cn(cx+d)}{\varDelta}g(M(x))+h_{n}.\label{HH}
\end{align}
 Notice that in terms of the Möbius-transformed polynomials $\varPhi(x)$
and $\varGamma(x)$ given by (\ref{PhiGamma}), we can rewritten (\ref{FF}),
(\ref{GG}) and (\ref{HH}) as, 
\begin{align}
\mathcal{F}(x) & =\frac{(cx+d)^{2}}{\varDelta^{2}}\varPhi(x),\\
\mathcal{G}_{n}(x) & =-\frac{2c(n-1)(cx+d)}{\varDelta^{2}}\varPhi(x)+\frac{(cx+d)}{\varDelta}\varGamma(x),\\
\mathcal{H}_{n}(x) & =\frac{c^{2}n(n-1)}{\varDelta^{2}}\varPhi(x)-\frac{cn}{\varDelta}\varGamma(x)+h_{n},
\end{align}
which proves that $\mathcal{F}(x)$, $\mathcal{G}_{n}(x)$ and $\mathcal{H}_{n}(x)$
are polynomials with degree at most $4$, $3$ and $2$, respectively.
This concludes the proofs.
\end{proof}
At first glance, Theorem \ref{PropDiff} seems to contradicts Routh-Bochner-Brenke's
Theorem because we found polynomial solutions of a second-order differential
equation whose coefficients are polynomials of degree greater than
$2$. However, this contradiction is only apparent because of the
following reason: while Routh-Bochner-Brenke's theorem concerns with
the solutions of the second-order differential equation (\ref{ODEP}),
where only constant term $h_{n}$ is supposed to be dependent of $n$,
both the coefficients $G_{n}(x)$ and $H_{n}(x)$ in the differential
equation (\ref{ODEQ}) can depend on $n$, so that (\ref{ODEP}) and
(\ref{ODEQ}) are actually of different forms. Furthermore, it is
clear that no restriction whatsoever can be imposed on the degrees
of the polynomial coefficients of a second-order differential equation
when all these coefficients are allowed to depend on $n$. 

\subsection{Weight function differential equations}

The weight function $w(x)$ of any classical orthogonal polynomial
also satisfy a differential equation. In this case, however, we have
the following first-order homogeneous differential equation \citep{Szego1939,Chihara2011}:
\begin{equation}
\frac{w'(x)}{w(x)}=\frac{g(x)-f'(x)}{f(x)},\label{PearsonP}
\end{equation}
where $f(x)$ and $g(x)$ are the same polynomials appearing in (\ref{ODEP}).
In fact, (\ref{PearsonP}) is obtained as a condition to the differential
equation (\ref{ODEP}) be written in the self-adjoint form.
\begin{equation}
\frac{\mathrm{d}}{\mathrm{d}x}\left[p(x)\frac{\mathrm{d}y(x)}{\mathrm{d}x}\right]+q_{n}(x)y(x)=0.\label{SLDE}
\end{equation}
It is well-known that any second-order differential equation can be
written in the self-adjoint form by multiplying it by some integrating
factor $w(x)$. In fact, multiplying (\ref{ODEP}) by $w(x)$ and
comparing with (\ref{SLDE}) we conclude that $p(x)=w(x)f(x)$, $p'(x)=w(x)g(x)$
and $q_{n}(x)=w(x)h(x)$, so that the relation $w(x)g(x)=\left[w(x)f(x)\right]'$
must hold, from which (\ref{PearsonP}) immediately follows. Besides,
whenever we have
\begin{equation}
\lim_{x\rightarrow a}f(x)w(x)x^{n}=0,\qquad\text{and}\qquad\lim_{x\rightarrow b}f(x)w(x)x^{n}=0,\qquad n\in\mathbb{N},\label{Lim}
\end{equation}
Sturm-Liouville theory will ensure that the sequence $P=\left\{ P_{n}(x)\right\} _{n=0}^{\infty}$
of polynomials (which we suppose to be the solutions of a well-behaved
Sturm-Liouville problem), will indeed be an orthogonal polynomial
sequence in the interval $I=\left(a,b\right)$ of the real line, with
respect to the weight function $w(x)$.

In this section, we shall show that weight functions of the Möbius-transformed
rational functions and polynomials also satisfy similar first-order
differential equations. 
\begin{thm}
\label{PropW}Let $w(x)$ be the weight function of a classical orthogonal
polynomial sequence $P=\left\{ P_{n}(x)\right\} _{n=0}^{\infty}$,
so that it satisfies the homogeneous first-order differential equation
(\ref{PearsonP}). Then, the corresponding weight function $\omega(x)$,
given by (\ref{omega}), of the sequence $R=\left\{ R_{n}(x)\right\} _{n=0}^{\infty}$
of Möbius-transformed rational functions, as defined in (\ref{Rn}),
will satisfy the following first-order differential equation: 
\begin{equation}
\frac{\omega'(x)}{\omega(x)}=\frac{G(x)-F'(x)}{F(x)},\label{PearsonR}
\end{equation}
where $F(x)$ and $G(x)$ are the same coefficients that appear in
the differential equation (\ref{ODER}). Similarly, the degree-dependent
weight function $\omega_{m,n}(x)$ given in (\ref{OCQ}), that is
associated with the sequence $Q=\left\{ Q_{n}(x)\right\} _{n=0}^{\infty}$
of Möbius-transformed polynomials, as defined in (\ref{Qn}), will
satisfy the following first-order differential equation: 
\begin{equation}
\frac{\omega_{m,n}'(x)}{\omega_{m,n}(x)}=\frac{1}{\mathcal{F}(x)}\left[\frac{\mathcal{G}_{m}(x)+\mathcal{G}_{n}(x)}{2}-\mathcal{F}'(x)\right],\label{PearsonQ}
\end{equation}
where $\mathcal{F}(x)$, $\mathcal{G}_{m}(x)$ and $\mathcal{G}_{n}(x)$
are the same polynomials defined in (\ref{FF}) and (\ref{GG}), respectively.
\end{thm}
\begin{proof}
Inverting relation (\ref{omega}), we can express $w(x)$ in terms
of $\omega(W(x))$: 
\begin{equation}
w(x)=\frac{\omega(W(x))}{\left(\frac{\mathrm{d}M(W(x))}{\mathrm{d}W(x)}\right)}.\label{PearsonW}
\end{equation}
Thereby, taking the derivative of $w(x)$, we obtain, 
\begin{equation}
w'(x)=\frac{\mathrm{d}\omega(W(x))}{\mathrm{d}W(x)}\frac{W'(x)}{\left(\frac{\mathrm{d}M(W(x))}{\mathrm{d}W(x)}\right)}-\omega(W(x))W'(x)\frac{\left(\frac{\mathrm{d}^{2}M(W(x))}{\mathrm{d}W(x)^{2}}\right)}{\left(\frac{\mathrm{d}M(W(x))}{\mathrm{d}W(x)}\right)^{2}}.\label{PearsonWPrime}
\end{equation}
Inserting (\ref{PearsonW}) and (\ref{PearsonWPrime}) into (\ref{PearsonP})
we shall get, after simplification, 
\begin{equation}
w'(x)=\left[\frac{\left(\frac{\mathrm{d}\omega(W(x))}{\mathrm{d}W(x)}\right)}{\omega(W(x))}-\frac{\left(\frac{\mathrm{d}^{2}M(W(x))}{\mathrm{d}W(x)^{2}}\right)}{\left(\frac{\mathrm{d}M(W(x))}{\mathrm{d}W(x)}\right)}\right]W'(x)=\frac{g(x)-f'(x)}{f(x)}.
\end{equation}
Now, making the change of variables $x=M(y)$, we shall find, 
\begin{equation}
\left\{ \frac{\omega'(y)}{\omega(y)}-\frac{M''(y)}{M'(y)}\right\} \frac{\mathrm{d}W(M(y))}{\mathrm{d}M(y)}=\frac{g(M(y))-\left(\frac{\mathrm{d}f(M(y))}{\mathrm{d}M(y)}\right)}{f(M(y))},
\end{equation}
 so that, 
\begin{equation}
\frac{\omega'(y)}{\omega(y)}=\left[\frac{g(M(y))-\left(\frac{\mathrm{d}f(M(y))}{\mathrm{d}M(y)}\right)}{f(M(y))}\right]M'(y)+\frac{M''(y)}{M'(y)}.
\end{equation}
Now, using the relations 
\[
M'(x)=\frac{\varDelta}{(cy+d)^{2}},\qquad M''(x)=-\frac{2c\varDelta}{(cy+d)^{3}},\qquad\text{and}\qquad\frac{\mathrm{d}f(M(y))}{\mathrm{d}M(y)}=\frac{1}{M'(y)}\frac{\mathrm{d}f(M(y))}{\mathrm{d}y},
\]
we can verify, taking (\ref{RFGH}) into account, that 
\begin{equation}
\left[\frac{g(M(y))-\left(\frac{\mathrm{d}f(M(y))}{\mathrm{d}M(y)}\right)}{f(M(y))}\right]M'(y)+\frac{M''(y)}{M'(y)}=\frac{G(y)-F'(y)}{F(y)},
\end{equation}
which establishes (\ref{PearsonR}). Finally, to prove (\ref{PearsonQ})
we can proceed as follows: from (\ref{OCQ}), we know that $\omega(x)=\omega_{m,n}(x)(cy+d)^{n+n}$.
Thus, inserting this expression into (\ref{PearsonR}) and computing
the derivative, we obtain, 
\begin{equation}
\frac{\omega_{m,n}'(x)}{\omega_{m,n}(x)}+\frac{c\left(m+n\right)}{(cx+d)}=\frac{G(x)-F'(x)}{F(x)}.
\end{equation}
Now, from (\ref{FF}) we get that $F(x)=\mathcal{F}(x)$ and, hence,
$F'(x)=\mathcal{F}'(x)$ as well; besides, from (\ref{GG}) we can
see that the term $c\left(m+n\right)/(cx+d)$ is exactly canceled
by expressing $G(x)$ in terms of $\frac{1}{2}\left[\mathcal{G}_{m}(x)+\mathcal{G}_{n}(x)\right]$,
from which (\ref{PearsonQ}) is obtained. This conclude the proof.
\end{proof}

\subsection{Rodrigues' type formulas}

Another interesting property of the classical orthogonal polynomial
sequences is that they can be defined through the so-called Rodrigues'
formulas \citep{Szego1939,Chihara2011}. In fact, if $P=\left\{ P_{n}(x)\right\} _{n=0}^{\infty}$
is such a sequence, then the polynomials $P_{n}(x)\in P$ can be found
by taking derivatives: 
\begin{equation}
P_{n}(x)=\frac{\epsilon_{n}}{w(x)}\frac{\mathrm{d}^{n}}{\mathrm{d}x^{n}}\left[f(x)^{n}w(x)\right],\label{RodP}
\end{equation}
 where $\epsilon_{n}$ are some standardizing constants, $w(x)$ the
corresponding weight function and $f(x)$ the same polynomial appearing
in the differential equation (\ref{ODEP}).

In this section, we shall show that the Möbius-transformed rational
functions and polynomials also can be found by taking derivatives
through a Rodrigues'-type formula.
\begin{thm}
Let $P=\left\{ P_{n}(x)\right\} _{n=0}^{\infty}$ be a sequence of
orthogonal polynomials on the real line satisfying a Rodrigues' type
formula as given in (\ref{RodP}). Then, the rational functions $R_{n}(x)$,
as defined in (\ref{Rn}), and the Möbius-transformed polynomials
$Q_{n}(x)$, as defined in (\ref{Qn}), will satisfy, respectively,
the following Rodrigues' type formulas: 
\begin{equation}
R_{n}(y)=\frac{\varDelta^{2n}}{\left(cx+d\right)^{2}}\frac{\epsilon_{n}}{\omega(x)}D_{x}^{n}\left[\frac{F(x)^{n}}{\left(cx+d\right)^{4n-2}}\omega(x)\right],\label{RodR}
\end{equation}
 and 
\begin{equation}
Q_{n}(y)=\frac{\varDelta^{2n}}{\left(cx+d\right)^{n+2}}\frac{\epsilon_{n}}{\omega_{n,n}(x)}D_{x}^{n}\left[\frac{\mathcal{F}(x)^{n}}{\left(cx+d\right)^{2n-2}}\omega_{n,n}(x)\right],\label{RodQ}
\end{equation}
where $F(x)=\mathcal{F}(x)$ is the same function defined in (\ref{RFGH}),
$\omega(x)$ and $\omega_{n,n}(x)$ are the new weight functions defined
in (\ref{omega}) and (\ref{OCQ}), respectively, and we introduced
the differential operator, 
\begin{equation}
D_{x}=\frac{1}{M'(x)}\frac{\mathrm{d}}{\mathrm{d}x}=\frac{\left(cx+d\right)^{2}}{\varDelta}\frac{\mathrm{d}}{\mathrm{d}x}.
\end{equation}
\end{thm}
\begin{proof}
Formulas (\ref{RodR}) and (\ref{RodQ}) follow in a straightforward
way by making the change of variable $x=M(y)$ into (\ref{RodP})
and using (\ref{Rn}), (\ref{Qn}), (\ref{omega}), (\ref{OCQ}) and
(\ref{RFGH}). 
\end{proof}
We remark that formulas (\ref{RodR}) and (\ref{RodQ}) can be rewritten
as well in the following equivalent form: 
\begin{align}
R_{n}(y) & =\frac{\varDelta}{\left(cy+d\right)^{2}}\frac{\epsilon_{n}}{\omega(y)}\left.\frac{\mathrm{d}^{n}}{\mathrm{d}x^{n}}\left[f^{n}(x)w(x)\right]\right|_{x=M(y)},\\
Q_{n}(y) & =\frac{\varDelta}{\left(cy+d\right)^{n+2}}\frac{\epsilon_{n}}{\omega_{n,n}(y)}\left.\frac{\mathrm{d}^{n}}{\mathrm{d}x^{n}}\left[f^{n}(x)w(x)\right]\right|_{x=M(y)}.
\end{align}
They can also be rewritten in many other ways --- for example, we
can expand the derivatives by invoking Faà-di-Bruno formula for the
$n$th derivative of a composed function, Leibniz formula for the
derivative of a product, etc ---, however, we shall not go through
these lines because the resulting expressions are very cumbersome.

\subsection{Generating functions}

As a last property, let us discuss the existence of generating functions
for the Möbius-transformed rational functions and polynomials. As
it is well-known, the sequences of classical orthogonal polynomials
can be generated by generating function of the following forms \citep{Szego1939,Chihara2011}:
\begin{equation}
\phi(x,t)=\sum_{k=0}^{\infty}P_{n}(x)t^{n},\qquad\text{or}\qquad\psi(x,t)=\sum_{k=0}^{\infty}P_{n}(x)\frac{t^{n}}{n!}.\label{GFP}
\end{equation}
 Similar formulas hold for the Möbius-transformed rational functions
and polynomials:
\begin{thm}
Let $P=\left\{ P_{k}(x)\right\} _{k=0}^{\infty}$ be a classical orthogonal
polynomial sequence whose generating function is one of those given
in (\ref{GFP}). Then, the corresponding generating functions of the
sequence $R=\left\{ R_{k}(x)\right\} _{k=0}^{\infty}$ of Möbius-transformed
rational functions $R_{n}(x)$ defined in (\ref{Rn}) are the following:
\begin{equation}
\varPhi(y,t)=\phi(M(y),t),\qquad\text{or}\qquad\varPsi(y,t)=\psi(M(y),t),
\end{equation}
respectively. Similarly, the corresponding generating functions of
the sequence $Q=\left\{ Q_{k}(x)\right\} _{k=0}^{\infty}$ of Möbius-transformed
polynomials $Q_{n}(x)$ defined in (\ref{Qn}) are: 
\begin{equation}
\varPhi(y,\tau)=\phi(M(y),(cy+d)\tau),\qquad\text{or}\qquad\varPsi(y,\tau)=\psi(M(y),(cy+d)\tau),
\end{equation}
 respectively.
\end{thm}
\begin{proof}
We have at once that, 
\begin{equation}
\varPhi(y,t)=\phi\left(M(y),t\right)=\sum_{k=0}^{\infty}P_{n}(M(y))t^{n}=\sum_{k=0}^{\infty}R_{n}(y)t^{n},
\end{equation}
 and
\begin{equation}
\varPsi(y,t)=\psi\left(M(y),t\right)=\sum_{k=0}^{\infty}P_{n}(M(y))\frac{t^{n}}{n!}=\sum_{k=0}^{\infty}R_{n}(y)\frac{t^{n}}{n!}.
\end{equation}
In the same manner, we have that, 
\begin{equation}
\varPhi(y,\tau)=\phi\left(M(y),(cy+d)\tau\right)=\sum_{k=0}^{\infty}P_{n}(M(y))(cy+d)^{n}\tau^{n}=\sum_{k=0}^{\infty}Q_{n}(y)\tau^{n},
\end{equation}
 and
\begin{equation}
\varPsi(y,\tau)=\psi\left(M(y),(cy+d)\tau\right)=\sum_{k=0}^{\infty}P_{n}(M(y))(cy+d)^{n}\frac{\tau^{n}}{n!}=\sum_{k=0}^{\infty}Q_{n}(y)\frac{\tau^{n}}{n!}.
\end{equation}
\end{proof}

\section{Applications\label{Section Applications}}

We close this paper with some applications of the theory developed
above. More specifically, we show that the sequences of Hermite, Laguerre,
Jacobi, Bessel and Romanovski polynomials are all related with each
other by a Möbius transformation. We also show that generalized Bessel
polynomials enjoy a finite orthogonality on the real line, as well
as the Romanovski polynomials satisfy an orthogonality relation on
the imaginary axis.

\subsection{From Laguerre to Bessel polynomials}

The \emph{associated Laguerre polynomials} $L_{n}^{\alpha}(x)$ are
polynomials orthogonal on the interval $I=\left(0,\infty\right)$
of the real line with respect to the weight function $w(x)=x^{\alpha}\mathrm{e}^{-x}$,
for any $\alpha>-1$ \citep{Szego1939,Chihara2011}. Performing the
re-scaling transformation $M(x)=\beta x$, we can easily generate
the so-called \emph{generalized Legendre polynomials} $L_{n}^{\left(\alpha,\beta\right)}(x)=L_{n}^{\alpha}(\beta x)$,
which are orthogonal on the same interval $I=\left(0,\infty\right)$
of the real line, for $\alpha>-1$ and $\beta>0$, but with respect
the weight function $w(x)=x^{\alpha}\mathrm{e}^{-\beta x}$. Now let
us consider the generalized Laguerre polynomials under the inversion
map $M(x)=1/x$. The resulting transformed polynomials will have,
therefore, the following form: 
\begin{equation}
\mathcal{L}_{n}^{\left(\alpha,\beta\right)}(x)=x^{n}L_{n}^{\left(\alpha,\beta\right)}(1/x).
\end{equation}
From the results of Section \ref{MTOP}, it follows that these polynomials
are also orthogonal on the interval $I=\left(0,\infty\right)$ but
with respect to a varying measure, provided that $\alpha>-1$ and
$\beta>0$. Namely, we have the following:
\[
\int_{0}^{\infty}\frac{\mathcal{L}_{m}^{\left(\alpha,\beta\right)}(x)\mathcal{L}_{n}^{\left(\alpha,\beta\right)}(x)}{x^{m+n}}\left(\frac{1}{x}\right)^{\alpha+2}\mathrm{e}^{-\beta/x}\mathrm{d}x=\frac{1}{\beta^{\alpha+1}}\frac{\varGamma\left(n+1+\alpha\right)}{\varGamma\left(n+1\right)}\delta_{m,n},\qquad\alpha>-1,\qquad\beta>0.
\]
In the following, we shall show that the polynomials $\mathcal{L}_{n}^{\left(\alpha,\beta\right)}(x)$
are related with the so-called \emph{Bessel polynomials}, which were
studied extensively in \citep{KrallFrink1949}. To this end, notice
that, according to the results of Section \ref{Subsection ODE}, the
polynomials $\mathcal{L}_{n}^{\left(\alpha,\beta\right)}(x)=y(x)$
will satisfy the differential equation: 
\begin{equation}
x^{2}y''(x)+\left[\beta+\left(1-2n-\alpha\right)x\right]y'(x)+n\left(n+\alpha\right)y(x)=0.\label{ODEL}
\end{equation}
Comparing with the differential equation given in for \citep{KrallFrink1949}
the generalized Bessel polynomials $\mathcal{B}_{n}^{\left(\gamma,\beta\right)}$,
namely, 
\begin{equation}
x^{2}y''(x)+\left(\beta+\gamma x\right)y'(x)+n\left(n+\gamma-1\right)y(x)=0,\label{ODEB}
\end{equation}
we see that a complete agreement between (\ref{ODEL}) and (\ref{ODEB})
is achieved when we set $\alpha=1-2n-\gamma$. Thus, the generalized
Bessel polynomials are given by the formula 
\begin{equation}
\mathcal{B}_{n}^{\left(\gamma,\beta\right)}(x)=\mathcal{L}_{n}^{\left(1-2n-\gamma,\beta\right)}(x),
\end{equation}
up to a normalization. Thereby, each polynomial $\mathcal{B}_{n}^{\left(\gamma,\beta\right)}(x)$
is associated with a generalized Laguerre polynomial with a different
value of $\alpha$. From the discussion given in Section \ref{Subsection OR},
we can conclude that the polynomials $\mathcal{B}_{n}^{\left(\gamma,\beta\right)}(x)$
will be satisfy the orthogonality relation
\begin{equation}
\int_{0}^{\infty}\mathcal{B}_{m}^{\left(\gamma,\beta\right)}(x)\mathcal{B}_{n}^{\left(\gamma,\beta\right)}(x)\Omega_{m,n}(x)\mathrm{d}x=\beta^{2n-2+\gamma}\frac{\varGamma\left(2-n-\gamma\right)}{\varGamma\left(n+1\right)}\delta_{m,n},\qquad\Omega_{m,n}(x)=x^{|m-n|+\gamma-3}\mathrm{e}^{-\beta/x},\label{ORBessel}
\end{equation}
provided that $\gamma<2-2\max(m,n)$ and $\beta>0$. From this we
can see that the polynomials $\mathcal{B}_{m}^{\left(\gamma,\beta\right)}(x)$
can at most enjoy a finite orthogonality on the interval $I=\left(0,\infty\right)$
of the real line\footnote{Notice as well that, for $\gamma$ a negative integer, the following
identity holds: $\mathcal{B}_{n}^{\left(\gamma,\beta\right)}(x)=\mathcal{B}_{1-\gamma-n}^{\left(\gamma,\beta\right)}(x)$.
Thus, in these cases the sequence $B=\{\mathcal{B}_{n}^{\left(\gamma,\beta\right)}(x)\}_{n=0}^{\infty}$
of generalized Bessel polynomials will be a defective sequence of
orthogonal polynomials.}. 

\begin{table}[H]
\centering%
\begin{tabular}{ll}
\hline 
Generalized Bessel polynomials & Bessel polynomials\tabularnewline
\hline 
$\mathcal{B}_{0}^{\left(\gamma,\beta\right)}(x)=1$ & $\mathcal{B}_{0}(x)=1$\tabularnewline
$\mathcal{B}_{1}^{\left(\gamma,\beta\right)}(x)=1+\frac{\gamma}{\beta}x$ & $\mathcal{B}_{1}(x)=1+x$\tabularnewline
$\mathcal{B}_{2}^{\left(\gamma,\beta\right)}(x)=1+\frac{2(1+\gamma)}{\beta}x+\frac{(1+\gamma)(2+\gamma)}{\beta^{2}}x^{2}$ & $\mathcal{B}_{2}(x)=1+3x+3x^{2}$\tabularnewline
$\mathcal{B}_{3}^{\left(\gamma,\beta\right)}(x)=1+\frac{3(2+\gamma)}{\beta}x+\frac{3(2+\gamma)(3+\gamma)}{\beta^{2}}x^{2}+\frac{(2+\gamma)(3+\gamma)(4+\gamma)}{\beta^{3}}x^{3}$ & $\mathcal{B}_{3}(x)=1+6x+15x^{2}+15x^{3}$\tabularnewline
\hline 
\end{tabular}

\caption{The first generalized Bessel polynomials and Bessel polynomials (normalized
with the constant terms equal to unity).}
 \label{TableBessel}
\end{table}

Following \citep{KrallFrink1949}, the usual Bessel polynomials are
obtained by setting $\gamma=2$ and $\beta=2$, from which we get
that $B_{n}(x)=\mathcal{L}_{n}^{\left(-1-2n,2\right)}(x)$, up to
a normalization, while the weight function becomes $\Omega_{m,n}(x)=x^{|m-n|-1}\mathrm{e}^{-2/x}$.
In this case, however, the condition for the validity of the orthogonality
relation (\ref{ORBessel}) would hold only for $n<0$, which is not
possible. Thus, we conclude that the Bessel polynomials are not orthogonal
on the real line. The first generalized and usual Bessel polynomials
are presented in Table \ref{TableBessel}.

\subsection{From Jacobi to Romanovski polynomials}

\emph{Romanovski polynomials }consist of an infinite sequence of real
polynomials studied by Romanovski in \citep{Romanovski1929}, although
they have appeared before in Routh's work \citep{Routh1885}. Romanovski
polynomials also appear in physical problems, which justified a more
detailed study of them --- see \citep{Alvarez2007,Raposo2007} and
references therein. In what follows we shall show that Romanovski
polynomials are also related with Jacobi polynomials $J_{n}^{\left(\alpha,\beta\right)}(x)$
via a Möbius transformation. In this case, however, we should consider
the mapping $M(x)=ix$ with the following values for the Möbius parameters:
$a=1$, $b=0$, $c=0$ and $d=-i$. 

Remember that Jacobi polynomials $J_{n}^{\left(\alpha,\beta\right)}(x)$
are orthogonal on the interval $I=(-1,1)$ of the real line with respect
to the weight function $w(x)=\left(1-x\right)^{\alpha}\left(1+x\right)^{\beta}$
for $\alpha>-1$ and $\beta>-1$ \citep{Szego1939,Chihara2011}. Thus,
according with the results of Section \ref{MTOP}, the Möbius-transformed
polynomials 
\begin{equation}
\mathcal{J}_{n}^{\left(\alpha,\beta\right)}(x)=\left(-i\right)^{n}J_{n}^{\left(\alpha,\beta\right)}(ix)
\end{equation}
will be orthogonal on the interval $J=\left(-i,i\right)$ along the
imaginary line, provided that $\alpha>-1$ and $\beta>-1$. The polynomials
$\mathcal{J}_{n}^{\left(\alpha,\beta\right)}(x)$, however, do not
have real coefficients for real $\alpha$ and $\beta$; nonetheless,
a real polynomial can be obtained if we allow $\alpha$ and $\beta$
to assume non-real values. The precise condition for this is that
$\alpha$ be the complex conjugate of $\beta$, so that we can write
$\alpha=\gamma+i\delta$, $\beta=\gamma-i\delta$ with $\gamma,\delta\in\mathbb{R}$
and $\delta\neq0$. In fact, up to a normalization, this gives place
to the \emph{Romanovski polynomials}\footnote{We remark that some authors adopt different values for the parameters
$\gamma$ and $\delta$ in the definition of Romanovski polynomials.}: 
\begin{equation}
\mathcal{R}_{n}^{\left(\gamma,\delta\right)}(x)=\mathcal{J}_{n}^{\left(\gamma+i\delta,\gamma-i\delta\right)}(x)=\left(-i\right)^{n}J_{n}^{\left(\gamma+i\delta,\gamma-i\delta\right)}(ix),
\end{equation}
The first Romanovski polynomials are presented in Table \ref{TableRomanovski}.

Now, in order to investigate the orthogonality of the polynomials
$\mathcal{R}_{n}^{\left(\gamma,\delta\right)}(x)$ along the imaginary
axis we need to verify under what conditions the Jacobi polynomials
with complex conjugate parameters are orthogonal on the real line.
It is not difficult to verify that Jacobi polynomials $J_{n}^{\left(\gamma+i\delta,\gamma-i\delta\right)}(x)$
are orthogonal on the interval $I=(-1,1)$ of the real line provided
that $\gamma>-1$ and $\delta\in\mathbb{R}$. In fact, this follows
because the limit relations (\ref{Lim}) are satisfied for any non-negative
integer $n$ whenever $\gamma>-1$ and $\delta$ any real number.
Besides, we can verify that all the moments $\mu(n)=\int_{-1}^{-1}x^{n}\left(1-x\right)^{\gamma+i\delta}\left(1+x\right)^{\gamma-i\delta}\mathrm{d}x$
exist under these conditions as well, which is enough to guarantee
the convergence of the integrals (\ref{OCP}) for any $m,n\in\mathbb{N}$,
for the same range of $\gamma$ and $\delta$. 

\begin{table}[H]
\centering%
\begin{tabular}{l}
\hline 
Romanovski polynomials\tabularnewline
\hline 
$\mathcal{R}_{0}^{\left(\gamma,\delta\right)}(x)=1$\tabularnewline
$\mathcal{R}_{1}^{\left(\gamma,\delta\right)}(x)=\delta+(1+\gamma)x$\tabularnewline
$\mathcal{R}_{2}^{\left(\gamma,\delta\right)}(x)=\frac{1}{4}\left(2+\gamma+2\delta^{2}\right)+\frac{1}{2}\delta(3+2\gamma)x+\frac{1}{4}(2+\gamma)(3+2\gamma)x^{2}$\tabularnewline
\hline 
\end{tabular}

\caption{The first Romanovski polynomials $\mathcal{R}_{n}^{\left(\gamma,\delta\right)}(x)=\left(-i\right)^{n}J_{n}^{\left(\gamma+i\delta,\gamma-i\delta\right)}(ix)$.}

\label{TableRomanovski}
\end{table}

Consequently, we conclude that Romanovski polynomials $\mathcal{R}_{n}^{\left(\gamma,\delta\right)}(x)$
will be orthogonal on the interval $J=(-i,i)$ of the imaginary axis
whenever $\gamma>-1$ and $\delta$ is a real number\footnote{Notice that the zeros of Romanovski polynomials $\mathcal{R}_{n}^{\left(\gamma,\delta\right)}(x)$
usually do not lie on the imaginary axis, which can be explained by
the fact that the Jacobi polynomials $J_{n}^{\left(\gamma+i\delta,\gamma-i\delta\right)}(x)$
are not real polynomials for $\gamma$ and $\delta$ real, so that
their zeros usually do not lie on the real line either. }. The corresponding weight function can be found either directly from
(\ref{omega}) and (\ref{OCQ}) or by solving the differential equation
(\ref{PearsonQ}). We have that, up to a multiplicative constant, 

\[
\omega(x)=\left(1-ix\right)^{\gamma+i\delta}\left(1+ix\right)^{\gamma-i\delta}=\left(x^{2}+1\right)^{\gamma}\mathrm{e}^{2\delta\arctan(x)}.
\]
Thus, the orthogonality relation reads: 
\begin{equation}
\int_{-i}^{i}\mathcal{R}_{m}^{\left(\gamma,\delta\right)}(x)\mathcal{R}_{n}^{\left(\gamma,\delta\right)}(x)\left(x^{2}+1\right)^{\gamma}\mathrm{e}^{2\delta\arctan(x)}\mathrm{d}x=\frac{2^{2\gamma+1}}{2n+2\gamma+1}\frac{\varGamma(n+1+\gamma+i\delta)\varGamma(n+1+\gamma-i\delta)}{\varGamma(n+1+2\gamma)\varGamma(n+1)}\delta_{m,n}.
\end{equation}
which holds for $\gamma>-1$ and real $\delta$. 

Finally, it should be mentioned that the limit relations (\ref{Lim}),
if applied to this Möbius-transformed weight function $\omega(x)$,
provide as well a finite orthogonality for the Romanovski polynomials
on the real line. In fact, in this case we can verify that only those
polynomials $\mathcal{R}_{m}^{\left(\gamma,\delta\right)}(x)$ and
$\mathcal{R}_{n}^{\left(\gamma,\delta\right)}(x)$ satisfying the
relation $m+n+2\gamma<-1$ will form an orthogonal polynomial sequence
on the interval $J=(-\infty,\infty)$ of the real line. 

\subsection{From Jacobi to Laguerre polynomials}

The Möbius transformation is rich enough to map any finite point to
the infinite --- in fact, on the extended complex plane $\mathbb{C}_{\infty}$
there is no special distinction between the infinite and any other
point. Thereby, we can transform a given finite interval of the real
line to a semi-infinite one. Following this idea, we shall show in
the sequel that Laguerre polynomials can be obtained from Jacobi polynomials
through a Möbius transformation followed by an adequate limit of their
parameters. 

The associated Laguerre polynomials are orthogonal polynomials on
the interval $I=(0,\infty)$ of the real line with respect to the
weight function $w(x)=x^{\alpha}\mathrm{e}^{-x}$ \citep{Szego1939,Chihara2011}.
Thus, let us define the the Möbius transformed polynomials $\mathcal{J}_{n}^{(\alpha,\beta)}(x)=\left(cx+d\right)^{n}J_{n}^{(\alpha,\beta)}\left(M(x)\right)$
and look for the transformation $W(x)$ that maps the points $l=-1$
and $r=1$ respectively to $\lambda=0$ and $\rho=\infty$. From (\ref{Mx}),
we plainly see that we should have $c=a$ and $d=-b$ to this end.
With this, we have that the first Möbius-transformed polynomial becomes:
\[
\mathcal{J}_{1}^{(\alpha,\beta)}(x)=a(\alpha+1)x+b(\beta+1).
\]
 Comparing this expression with the first Laguerre polynomial, $L_{1}^{\beta}(x)=\beta+1-x$,
we see that we must have $a=-1/\left(\alpha+1\right)$ and $b=1$.
Then we can verify that the other Möbius transformed polynomials $\mathcal{J}_{n}^{(\alpha,\beta)}(x)$
will match with the Laguerre polynomials $L_{n}^{\beta}(x)$ after
we take the limit $\alpha\rightarrow\infty$, that is, 
\begin{equation}
L_{n}^{\beta}(x)=\lim_{\alpha\rightarrow\infty}\left(\frac{x+\alpha+1}{\alpha+1}\right)^{n}\left(-1\right)^{n}J_{n}^{(\alpha,\beta)}\left(\frac{x-\alpha-1}{x+\alpha+1}\right)=\lim_{\alpha\rightarrow\infty}\left(-1\right)^{n}J_{n}^{(\alpha,\beta)}\left(\frac{x-\alpha-1}{x+\alpha+1}\right).
\end{equation}
 Using the identity $J_{n}^{(\beta,\alpha)}(x)=\left(-1\right)^{n}J_{n}^{(\alpha,\beta)}(-x)$
we can also write: 
\begin{equation}
L_{n}^{\alpha}(x)=\lim_{\beta\rightarrow\infty}\left(\frac{x+\beta+1}{\beta+1}\right)^{n}J_{n}^{(\alpha,\beta)}\left(\frac{\beta+1-x}{\beta+1+x}\right)=\lim_{\beta\rightarrow\infty}J_{n}^{(\alpha,\beta)}\left(\frac{\beta+1-x}{\beta+1+x}\right).
\end{equation}
 These expressions can be compared with Szeg\H{o}'s identity \citep{Szego1939}:
$L_{n}^{\alpha}(x)=\lim_{\beta\rightarrow\infty}J_{n}^{(\alpha,\beta)}\left(1-2\beta^{-1}x\right)$
. 

Finally, let us show that Laguerre's weight function can also be obtained
from Jacobi's weight function through the same procedure. Notice that
$\varOmega_{m,n}(x)\rightarrow\omega(x)$ as $\alpha\rightarrow\infty$
because $c\rightarrow0$ in this limit. Thus, is enough to prove that
$w(x)=\left(1-x\right)^{\alpha}\left(1+x\right)^{\beta}$ reduces
to $\omega(x)=x^{\beta}\mathrm{e}^{-x}$ in the limit $\alpha\rightarrow\infty$
after we perform the Möbius transformation above. From (\ref{PearsonR})
we can verify that the transformed weight function is, 
\begin{equation}
\omega(x)=x^{\beta}\left(x+\alpha+1\right)^{-\alpha-\beta-2}\kappa(\alpha,\beta),
\end{equation}
where $\kappa(\alpha,\beta)$ is the constant of integration. From
this we can see that Laguerre's weight function is obtained after
we set $\kappa(\alpha,\beta)=\alpha^{\beta+1}(\alpha+1)^{\alpha+1}$
and take the limit $\alpha\rightarrow\infty$. All the other properties
of Laguerre polynomials can be obtained from those of Jacobi by similar
arguments. 

\subsection{From Jacobi to Hermite polynomials }

Finally, we can show that Hermite polynomials can also be obtained
from Jacobi polynomials through a Möbius transformation followed by
a specific limit. At first sight, we would be tempted to map the finite
interval $I=(-1,1)$ into the infinite interval $J=(-\infty,\infty)$.
This, however, it is not possible to be done with a Möbius transformation
because the points $x=\pm\infty$ are regarded as equals in the extended
complex plane $\mathbb{C}_{\infty}$ and the Möbius transformation
is one-to-one and onto in $\mathbb{C}_{\infty}$. Thus, to work this
around, we shall proceed in a somewhat different way: the idea is
to first find the Möbius transformation $W(x)$ that maps the interval
$I=(-1,1)$ into the interval $J=(-\xi,\xi)$ for some $\xi>0$ and,
then, to take the limit $\xi\rightarrow\infty$. Imposing further
that the point $0$ is mapped to $0$, we find that the required Möbius
transformation becomes $W(x)=a\xi$, whose inverse is $M(x)=a/\xi$
($a$ is still arbitrary). 

Now, according (\ref{omega}), (\ref{OCQ}) or (\ref{PearsonQ}),
we get that Jacobi's weight function $w(x)=\left(1-x\right)^{\alpha}\left(1+x\right)^{\beta}$
will be transformed into 
\[
\omega(x)=\frac{1}{\xi}\left(1-\frac{x}{\xi}\right)^{\alpha}\left(1+\frac{x}{\xi}\right)^{\beta}\kappa(\alpha,\beta),
\]
where $\kappa(\alpha,\beta)$ is the constant of integration. From
this we can plainly see that Hermite's weight function $\omega(x)=\mathrm{e}^{-x^{2}}$
is achieved by making $\beta=\alpha$, $\kappa(\alpha,\beta)=\xi=\sqrt{\alpha}$
and taking limit $\alpha\rightarrow\infty$ (so that we get $\xi\rightarrow\infty$
as well). With these identifications, the Möbius-transformed polynomials
become $\mathcal{J}_{n}^{(\alpha)}(x)=\left(a^{n}\sqrt{\alpha}\right)^{n}J_{n}^{(\alpha,\alpha)}\left(x/\sqrt{\alpha}\right)$.
Now, the limits of $\mathcal{J}_{n}^{(\alpha)}(x)$ for $\alpha\rightarrow\infty$
only exist if $a$ is inversely proportional to $\alpha$. In fact,
if we set $a=2/\alpha$ and take the limit $\alpha\rightarrow\infty$
then we can verify that we get a complete correspondence between the
Möbius-transformed polynomials $\mathcal{J}_{n}^{(\alpha)}(x)$ and
Hermite polynomials $H_{n}(x)$: 
\begin{equation}
\frac{1}{n!}H_{n}(x)=\lim_{\alpha\rightarrow\infty}\mathcal{J}_{n}^{(\alpha)}(x)=\lim_{\alpha\rightarrow\infty}\left(\frac{2}{\sqrt{\alpha}}\right)^{n}J_{n}^{(\alpha,\alpha)}\left(\frac{x}{\sqrt{\alpha}}\right).
\end{equation}

\section*{Acknowledgments}

We kindly thank Professors C. F. Bracciali and A. Sri Ranga for the
discussions we had about these ideas and also for their valuable comments.
The work of RSV was supported by grants from Coordination for the
Improvement of Higher Education Personnel (CAPES). The work of VB
was supported by funds from São Paulo Research Foundation (FAPESP),
grant \#2016/02700-8. 

\section*{Bibliography}

\bibliographystyle{elsarticle-num}
\bibliography{MobiusP}

\end{document}